\documentclass[reqno,11pt]{amsart}
\usepackage{etex}
\usepackage{tikz}
\usetikzlibrary{matrix,arrows}

\usepackage{amsmath,amssymb,amscd,amsxtra,dsfont}
\usepackage{graphicx,pst-grad,pst-plot,pst-coil}
\usepackage{pstricks}
\usepackage{eucal}
\usepackage[margin=1.5in]{geometry}

\usepackage[colorlinks=true,linkcolor=blue,citecolor=blue]{hyperref}  

\usepackage{setspace}

\usepackage{mathptmx}
\DeclareSymbolFont{txsymbols}{OMS}{txsy}{m}{n}
\SetSymbolFont{txsymbols}{bold}{OMS}{txsy}{bx}{n}

\DeclareSymbolFont{txlargesymbols}{OMX}{txex}{m}{n}
\SetSymbolFont{txlargesymbols}{bold}{OMX}{txex}{bx}{n}

\let\amalg\relax
\DeclareMathSymbol{\amalg}{\mathbin}{txsymbols}{113}
\let\coprod\relax
\DeclareMathSymbol{\coprod}{\mathop}{txlargesymbols}{96}
\DeclareMathSymbol{\nabla}{\mathord}{txsymbols}{114}

\input diagxy
\input xy
\xyoption{all}
\def\bfig{\vcenter\bgroup\xy}
\def\efig{\endxy\egroup}

\theoremstyle{plain}

\newtheorem{thm}{Theorem}
\newtheorem{proposition}[thm]{Proposition}
\newtheorem{corollary}[thm]{Corollary}
\newtheorem{lemma}[thm]{Lemma}

\theoremstyle{definition}

\newenvironment{definition}{\refstepcounter{thm}
{\medskip\par\noindent\bf Definition \arabic{section}.\arabic{thm}.
}}{\vskip 2ex\par}

\newenvironment{example*}{
  {\medskip\noindent\it Example.}}{\vskip 2ex\par}

\newenvironment{conjecture}{\refstepcounter{thm}
{\medskip\par\noindent\bf Conjecture \arabic{section}.\arabic{thm}.
}}{\vskip 2ex\par}

\renewenvironment{proof}{
{\medskip\par\noindent\sc Proof.\ }}{\vskip 2ex\par}

\newcommand{\cala}{\mathcal{A}}

\newcommand{\calo}{\mathcal{O}}

\newcommand{\comp}{\mathrel{\scriptstyle\circ}}

\newcommand{\C}{\mathbb{C}}

\newcommand{\cinf}{{C^{\infty}}}

\newcommand{\inv}{^{-1}}

\newcommand{\tr}{\operatorname{tr}}

\def \pf{\begin{proof}}
\def \epf{\end{proof}}
\def \enum{\begin{enumerate}}
\def \eenum{\end{enumerate}}

\renewcommand{\to}{\rightarrow}

\def \defi{\begin{definition}}
  \def \edefi{\end{definition}}

\def \prop{\begin{proposition}}
\def \eprop{\end{proposition}}

\def \lem{\begin{lemma}}
\def \elem{\end{lemma}}

\def \cor{\begin{corollary}}
\def \ecor{\end{corollary}}

\newcommand{\ex}{\begin{example*}}
\newcommand{\eex}{\end{example*}}

\newenvironment{exer*}
  {\small\begin{exercise}}
  {\end{exercise}}

\newcommand{\probref}[1]{\textbf{\ref{#1}} } 

\def \ex*{\begin{example*}}
\def \eex*{\end{example*}}

\newenvironment{remark*}{
{\medskip\noindent\it Remark.}}{\vskip 2ex\par}
\def \rem*{\begin{remark*}}
\def \erem*{\end{remark*}}

\newenvironment{claim*}{
{\medskip\noindent\it Claim.}}{\vskip 2ex\par}

\def \pf{\begin{proof}}
\def \epf{\end{proof}}

\def \enum{\begin{enumerate}}
\def \eenum{\end{enumerate}}

\numberwithin{equation}{section}
\numberwithin{figure}{section}
\numberwithin{thm}{section}

\def\C{\mathbb C}

\newcommand{\dbar}{\bar{\partial}}

\newcommand{\term}[1]{\textbf{\textit{#1}}}
 
\usepackage[colorlinks=true,linkcolor=blue]{hyperref}

\begin{document}
\title
{Lefschetz Fixed Point Theorems for Correspondences}

\author{Loring W. Tu}
\address{Department of Mathematics\\
Tufts University\\
Medford, MA 02155} 
\email{loring.tu@tufts.edu}
\urladdr{ltu.pages.tufts.edu}
\date{}
        \begin{center}
\textit{Dedicated to Catriona Byrne\\on the occasion of her retirement
  from Springer}
\end{center}
\begin{abstract}
The classical Lefschetz fixed point theorem states that the number of fixed points, counted with multiplicity $\pm 1$, of a smooth map $f$ from a manifold $M$ to itself can be calculated as the alternating sum $\sum (-1)^k \textrm{ tr } f^*|_{H^k(M)}$ of the trace of the induced homomorphism in cohomology.
In 1964, at a conference in Woods Hole, Shimura conjectured a Lefschetz fixed point theorem for a holomorphic map, which Atiyah and Bott proved and generalized into a fixed point theorem for elliptic complexes.  However, in Shimura's recollection, he had conjectured more than the holomorphic Lefschetz fixed point theorem.  He said he had made a conjecture for a holomorphic correspondence, but he could not remember the statement.
This paper is an exploration of Shimura's forgotten conjecture, first
for a smooth correspondence, then for a
holomorphic correspondence in the form of two conjectures and finally in the form of an open problem involving an
extension to holomorphic vector bundles over two varieties and the calculation of the
trace of a Hecke correspondence.
\end{abstract}
\maketitle

\bigskip
\bigskip


\setcounter{page}{1} \setcounter{thm}{0}

The classical Lefschetz fixed point theorem states that the number of
fixed points, counted with multiplicity $\pm 1$, of a smooth map $f$
from a compact oriented manifold $M$ to itself can be calculated as
the alternating sum $\sum (-1)^k \tr f^*|_{H^k(M)}$ of the trace of
the induced homomorphism in cohomology.\footnote{Throughout this
  article $H^*(M)$ denotes de Rham cohomology and the fixed points are
  assumed to be nondegenerate.} This alternating sum is called the
\term{Lefschetz number} $L(f)$ of the map $f$.
As a corollary, if the Lefschetz number $L(f)$ is nonzero, then $f$
has at least one fixed point. 

In 1964, at the AMS Woods Hole Conference in Algebraic Geometry,
Shimura conjectured an analogue for a holomorphic map of the Lefschetz
fixed point theorem.
Shimura's conjecture got the people at the conference all excited,
and there was a workshop to prove it.  At the end of the conference,
there were two proofs---an algebraic proof by Verdier, Mumford,
Hartshorne, and others, along more or less classical lines from the
Grothendieck version of Serre duality, and an analytic proof by Atiyah
and Bott.
Grothendieck generalized the algebraic proof in \cite[Cor.~6.12,
p.~131]{grothendieck--illusie}
and Atiyah and Bott generalized the analytic proof to the Atiyah--Bott fixed point
theorem for an elliptic complex in \cite[Th.~1, p.~246]{atiyah--bott66} and
\cite[Th.~A, p~377]{atiyah--bott67}.

There was a bit of controversy about this, because afterwards,
Shimura's name disappeared from this theorem.  It is now called the
holomorphic Lefschetz fixed point theorem and the more general
version is the Atiyah--Bott fixed point theorem.  Shimura was quite
upset about this.
The principals in this story have all passed away, Atiyah and Shimura
in the last two years.
Fortunately, while they were still living, I was able to interview
Michael Atiyah, Raoul Bott, Goro Shimura, and John Tate about the
holomorphic Lefschetz fixed point theorem and in 2015 I published an
article \cite{tu15} in the hope of setting the history straight.

In Shimura's recollection, he had conjectured more than the
holomorphic Lefschetz fixed point theorem.  He said he had made a
conjecture for an algebraic correspondence, which for a complex
projective variety is the same as a holomorphic correspondence, but he could not remember
the statement nor did he keep any notes.
He believed that his conjecture for a holomorphic correspondence
should have number-theoretic consequences for a Hecke correspondence
and higher-dimensional automorphic forms.
This article is an exploration of Shimura's forgotten conjecture,
first for a smooth correspondence, then 
for a holomorphic correspondence in the form of two conjectures, and finally an open problem
involving an
extension to holomorphic vector bundles over two varieties and the calculation of the
trace of a Hecke correspondence.

The coincidence locus of two set maps $f$, $g\colon N \to M$ is the
subset of $N$ on which they agree.
A coincidence locus is sometimes the fixed-point set of a
correspondence and vice versa, but the two types of sets are not the same.
In Lefschetz's original paper \cite{lefschetz26} he obtained a
coincidence locus formula for two continuous maps of manifolds.  The fixed-point formula
for a smooth correspondence (Theorem~\ref{t:lefschetz}) in this article agrees with
Lefschetz's coincidence formula when the coincidence is a
correspondence.
Thus, Theorem~\ref{t:lefschetz} is essentially already in Lefschetz
\cite{lefschetz26}.
It is also a special case of \cite{dell'ambrogio} for the trivial
group action and of \cite[Theorem 4.7, p.~15]{goresky--macpherson93} for the
trivial sheaf.
Since Lefschetz's time, there have been many generalizations and
variants of his coincidence and fixed-point formulas
(\cite{kuga--sampson},
\cite{goresky--macpherson93},
\cite{goresky--macpherson03}, \cite{dell'ambrogio}, \cite{taelman16}).
I offer this article in the hope that a simple-minded proof of a simple-minded
statement in the smooth case may spur some interest in the holomorphic case.

At the end of the article, I include as historical documents some
emails concerning the conjecture from Shimura to Atiyah and me in
2013.
I would like to thank Jeffrey D.\ Carlson, Mark Goresky, Jacob Sturm,
and the anonymous referee for many helpful comments and suggestions.

\section{Correspondences}

\begin{definition}
Let $X$ be a topological space.  A \term{correspondence} on $X$ is a
subspace $\Gamma \subset X \times X$ such that the two
projections $\pi_i\colon \Gamma \subset X \to X$, $i=1,2$, are
covering maps of finite degree (Figure~\ref{f:correspondence}).
\end{definition}

\begin{figure}
\begin{center}
{\psset{unit=.8}
\begin{pspicture}(-.5,-.5)(3.5,3.5)
\psline(-.5,0)(3.5,0)
\psline(0,-.5)(0,3.5)
\psline[linewidth=1pt](0.5,0.5)(3.5,3.5)
\pscurve[linewidth=2pt](0.5,3.5)(3.5,2.5)(0.5,1.5)(3.5,0.5)
\uput{.2}[0](3.5,3.5){$\Delta$}
\uput{.2}[0](3.5,2.5){$\Gamma $}
\uput{.2}[270](2,0){$X$}
\uput{.2}[180](0,2){$X$}
\end{pspicture}
}
\caption{A correspondence $\Gamma$ on $X$}
\label{f:correspondence}
\end{center}
\end{figure}
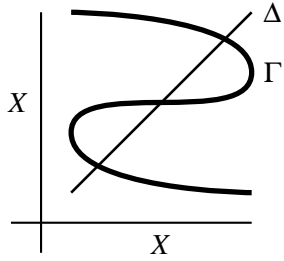

A correspondence $\Gamma$ on $X$ may be viewed as the graph of a
multivalued function from $X$ to $X$ whose value at $p \in X$ is the
set $\pi_2 \pi_1^{-1}( p )$.  By symmetry, it can also be
the multivalued function $\pi_1 \pi_2^{-1}$.

We have defined a correspondence in the continuous category.  Clearly,
it can also be defined in the categories of smooth manifolds and smooth
maps, complex manifolds and holomorphic maps, and algebraic varieties and
regular maps.

\section{Lefschetz Number of a Smooth Correspondence}

Suppose $\pi\colon N \to M$ is a $\cinf$ covering map of degree $r$.
Denote by $\cala^k(N)$ the vector space of smooth $k$-forms on $N$.
For $\omega \in \cala^k(N)$ and $p \in M$, define a $k$-covector
$(\pi_*\omega)_p$ at $p$ on $M$ by
\[
 (\pi_*\omega)_p(v_1, \ldots, v_k) = \sum_{q_i\in\pi\inv(p)} \omega_{q_i}(v_1^i,
   \ldots, v_k^i),
 \]
 where $v_1, \ldots, v_k \in T_pM$ and $v_1^i,\ldots, v_k^i$ are the
 unique tangent vectors in $T_{q_i}(N)$ such that $\pi_* v_j^i = v_j$.
 As $p$ varies over $M$, the $k$-covector $(\pi_*\omega)_p$ becomes a
 $k$-form $\pi_*\omega$ on $M$.
 This defines a pushforward map $\pi_*\colon \cala^k(N) \to
 \cala^k(M)$ of smooth $k$-forms on $N$.
 Since $\pi_*d = d\pi_*$, the pushforward induces a linear map
 $H^k(N)\to H^k(M)$ in cohomology, also denoted by $\pi_*$.
 
 A smooth correspondence induces a linear map on the cohomology of the
 manifold $M$ by
 \[
   \pi_{1*} \pi_2^*\colon H^*(M) \to H^*(M).
 \]

 \begin{definition}
   The \term{Lefschetz number} $L(\Gamma)$ of a smooth correspondence
   $\Gamma$ is defined to be the alternating sum of the traces of the linear map
   $ \pi_{1*}\pi_2^*$ on $H^k(M)$:
   \[
     L(\Gamma) = \sum_{k=0}^n (-1)^k \tr  \pi_{1*} \pi_2^*\colon H^k(M) \to
     H^k(M).
   \]
   \end{definition}

   \section{Fixed Points of a Smooth Correspondence}

   A \term{fixed point} of a smooth correspondence $\Gamma$ on a manifold $M$ is a
   point $p$ in $M$ such that $(p,p) \in \Gamma \cap \Delta$ in $M
   \times M$, where $\Delta$ is the diagonal.  The correspondence is
   called \term{transversal} if $\Gamma$ intersects $\Delta$
   transversally in $M \times M$.  In this case, the fixed points are
   said to be \term{nondegenerate}.
   Nondegenerate fixed points are isolated.

   When the manifold $M$ is oriented and the correspondence is
   transversal, we can assign a \term{multiplicity} or \term{index} to
   each fixed point $p$ in the usual way:  $\iota_{\Gamma}(p) = \pm 1$
   depending on whether the orientation on the tangent space
   $T_{(p,p)}(M \times M)$ agrees or disagrees with the orientation on
   the direct sum $T_{(p,p)}\Gamma \oplus T_{(p,p)} \Delta$.
   The intersection number $\#(\Gamma, \Delta)$ is then the sum
   $\sum \iota_{\Gamma}(p)$, where the sum runs over all fixed points
   $p$ of the correspondence $\Gamma$.
   When the manifold $M$ is compact, the number of nondegenerate fixed
   points is finite and the intersection number is defined.

     \section{The Trace of a Smooth Correspondence}

     We show how to calculate the trace of a correspondence in
     terms of differential forms.

     \begin{proposition} \label{p:trace}
       Let $\Gamma \subset M \times M$ be a smooth correspondence on
       a compact oriented
       smooth manifold $M$,  $\psi_1, \ldots, \psi_m$ closed
       $(n-k)$-forms on $M$ representing a basis for
       $H^{n-k}(M)$, and $\psi_1^*, \ldots, \psi_m^*$ closed
       $k$-forms representing the dual basis for $H^k(M)$.
       Then on $H^k(M)$,
       \[
         \tr \pi_{1*}\pi_2^* = \sum_{i=1}^m \int_{\Gamma}
         \pi_1^*\psi_i \wedge \pi_2^* \psi_i^*.
       \]
     \end{proposition}

     \begin{proof}
       Let $[a_j^i]$ be the matrix of the linear operator
       $\pi_{1*}\pi_2^*$ on $H^k(M)$:
       \[
         \pi_{1*}\pi_2^*(\psi_j^*) = \sum a_j^i \psi_i^*.
       \]
       Then
       \begin{alignat*}{2}
         a_j^i &= \int_M \psi_i \wedge \pi_{1*}\pi_2^*\psi_j^*\\
               &= \frac{1}{r} \int_{\Gamma} \pi_1^* \psi_i \wedge
               \pi_1^*\pi_{1*}\pi_2^*\psi_j^*&\quad&\left(\text{because }
               \int_M \tau= \frac{1}{r}\int_{\Gamma} \pi_1^*\tau\right)\\
                 &= \int_{\Gamma} \pi_1^*\psi_i \wedge \pi_2^*
                 \psi_j^*&\quad&\left(\text{because } \omega \wedge
                 \pi_1^*\pi_{1*} \tau = r \omega \wedge
                 \tau\right).
               \end{alignat*}
               Therefore,
               \[
                 \tr \pi_{1*}\pi_2^* = \sum_i a_i^i = \sum_i \int_{\Gamma} \pi_1^*\psi_i \wedge \pi_2^*
                 \psi_i^*. \tag*{\qed}
               \]    
\end{proof}

\section{The Lefschetz Fixed Point Theorem for a Smooth
  Correspondence}

  \begin{thm}[Lefschetz fixed point theorem for a smooth
     correspondence] \label{t:lefschetz}
     Suppose $\Gamma$ is a transversal smooth correspondence on a compact, oriented
     smooth $n$-manifold $M$.  Then the Lefschetz number of $\,\Gamma$ is
     \[
       L(f) = \sum_{\text{\rm fixed points}\ p} \iota_{\Gamma}(p).
     \]
   \end{thm}

   Our proof largely emulates the approach of Griffiths and Harris in
    their account of the Lefschetz fixed point formula for a smooth
   self-map \cite[Chap.~3, Sec.~4, pp.~419--422]{griffiths--harris}, but generalized to a smooth correspondence.
    The main idea is quite simple.  By Poincar\'e duality,
   the intersection number $\#(\Gamma,\Delta)$ of the correspondence $\Gamma$ with the
   diagonal $\Delta$ can be calculated as the integral of the wedge
   product of the differential forms representing their Poincar\'e
   duals.
   On the other hand, with the trace formula of Proposition~\ref{p:trace}, the Lefschetz number of the correspondence
   $\Gamma$ can also be calculated in terms of differential forms.
   The two expressions in differential forms turn out to be equal.

   \begin{proof}
   Let
$\psi_1, \ldots, \psi_s$ be closed
       forms on $M$ representing a basis for
       $H^*(M)$, and $\psi_1^*, \ldots, \psi_s^*$ closed
       forms representing the dual basis for $H^*(M)$.
       Note that the forms $\psi_i, \psi_j^*$ run over all degrees,
       but $\psi_i$ and $\psi_i^*$ have complementary degrees in $n$.
       By the K\"unneth formula, $\pi_1^* \psi_i \wedge \pi_2^*
       \psi_j$ represent
       a basis for the cohomology $H^*(M\times M)$.
        It is proven in \cite[p.~420]{griffiths--harris} that the Poincar\'e dual of the
       diagonal $\Delta$ is given by
       \[
         \eta_{\Delta} = \sum_i (-1)^{\deg \psi_i^*} \pi_1^* \psi_i
         \wedge \pi_2^* \psi_i^*.
       \]
       Then
     \begin{alignat*}{2}
       L(\Gamma) &= \sum_k (-1)^k \tr \pi_{1*}\pi_2^*|_{H^k(M)}\\
       &= \sum_k (-1)^k \sum_{\deg \psi_i = n-k} \int_{\Gamma}
         \pi_1^*\psi_i \wedge \pi_2^* \psi_i^* &\quad&(\text{Prop.}~\ref{p:trace})\\
         &= \int_{\Gamma} \sum_i (-1)^{\deg \psi_i^*}
         \pi_1^*\psi_i \wedge \pi_2^* \psi_i^* &\quad&(\text{$\psi_i$
           runs over all degrees})\\
         &=\int_{\Gamma} \eta_{\Delta} &\quad& (\text{by the formula
           for $\eta_{\Delta}$})\\
         &= \int_M \eta_{\Gamma} \wedge
         \eta_{\Delta}&\quad&(\text{def.\ of the Poincar\'e dual $\eta_{\Gamma}$})\\
         &=\#(\Gamma \cdot \Delta) = \sum_{\text{fixed points $p$}}
         \iota_{\Gamma}(p). \tag*{\qed}
       \end{alignat*}
     \end{proof}

     \section{A Conjecture for a Holomorphic Correspondence}

     Let $\Gamma$ be a \term{holomorphic correspondence} on a complex
     manifold $M$ of complex dimension $n$, that is, a complex
     submanifold  of $M \times M$ such that the two projections
     $\pi_i\colon \Gamma \to M$ are holomorphic covering maps.
     As for a smooth correspondence, a fixed point of the holomorphic correspondence $\Gamma$ is a
     point $p \in M$ such that $(p,p)$ is in the intersection $\Gamma
     \cap \Delta$
     in $M\times M$, where
     $\Delta$ is the diagonal in $M\times M$.
      The correspondence $\Gamma$ is said to be \term{transversal}
      if $\Gamma$ intersects the diagonal $\Delta$ transversally in $M
      \times M$.
     
     Denote by $\calo$ the sheaf of holomorphic functions
     and $\cala^{p,q}$ the sheaf of $\cinf$ $(p,q)$-forms on $M$.
     Let $\Gamma(M, \cala^{p,q})$ be the space of global sections of
     $\cala^{p,q}$; these are simpley the $\cinf$ $(p,q)$-forms on
     $M$.
     The sheaf $\calo$ has an acyclic resolution
     \[
       0 \to \calo \to \cala^{0,0} \overset{\dbar}{\longrightarrow}
       \cala^{0,1} \overset{\dbar}{\longrightarrow}
       \cala^{0,2} \overset{\dbar}{\longrightarrow} \cdots
     \]
     and the cohomology $H^k(M,\calo)$ is the cohomology of the
     differential complex of global sections
     \[
       \Gamma(M, \cala^{0,0}) \overset{\dbar}{\longrightarrow}
\Gamma(M, \cala^{0,1}) \overset{\dbar}{\longrightarrow}
\Gamma(M, \cala^{0,2}) \overset{\dbar}{\longrightarrow} \cdots .
\]
(For background on sheaf cohomology, see \cite{tu22}.)

For a holomorphic covering map $f\colon N \to M$, both the pullback
$f^*$ and the pushforward $f_*$ of $\cinf$ $(0,k)$-forms are cochain maps
of the complexes $\Gamma(N, \cala^{0,\bullet})$ and
$\Gamma(M,\cala^{0,\bullet})$.
Since the projection maps $\pi\colon \Gamma \to M$ are holomorphic
covering maps, both the pullback $\pi_2^*\colon H^*(M,\calo) \to
H^*(\Gamma, \calo)$ and the pushfoward $\pi_1^* H^*(\Gamma, \calo) \to
H^*(M,\calo)$ in cohomology are well-defined.
Thus, the holomorphic correspondence $\Gamma$ induces linear maps of
cohomology groups 
     \[
       \pi_{1*}\pi_2^*\colon H^k(M, \calo) \to H^k(M, \calo), \quad k=
       0, \ldots, n.
     \]
     The \term{holomorphic Lefschetz number} $L(\Gamma,\calo)$ of
     $\Gamma$ is defined to be an alternating sum of traces as
     before:
     \[
       L(\Gamma,\calo) = \sum_{k=0}^n (-1)^k \tr \pi_{1*}\pi_2^*\colon H^k(M,
       \calo) \to H^k(M, \calo).
     \]

     The holomorphic Lefschetz number is a global invariant.  Next we
     define the local contribution at each fixed point.  Since a
     correspondence is a holomorphic covering map of $M$ via $\pi_1$, locally it
     is the graph of a holomorphic function $f$.
     At a fixed point $p$, let $J(\Gamma)$ be the Jacobian matrix of
     the holomorphic function $f$ with respect to any holomorphic
     coordinate system.

     \begin{conjecture} \label{conj:1}
       {\it If $\Gamma$ is a transversal holomorphic correspondence
     on a compact complex manifold $M$, then the holomorphic Lefschetz number
     of $\Gamma$ is given by
     \[
       L(\Gamma,\calo) = \sum_{\emph{fixed points $p$}}\displaystyle
       \frac{1}{ 1 - \det J(\Gamma)_p}.
     \]
     }
   \end{conjecture}
     I do not have any evidence for this conjecture other than that it
     specializes to the correct formula when the correspondence
     $\Gamma$ is the graph of a holomorphic map $f\colon M \to M$.
     Of course, the simplicity of the statement plays in its favor.
    
\section{Extension to Holomorphic Vector Bundles}

In their seminal paper on the fixed point theorem for elliptic
complexes \cite{atiyah--bott68}, Atiyah and Bott extended, as a
corollary of their general theorem, the
Lefschetz fixed point theorem to a holomorphic vector bundle for a
self-map of a compact complex manifold.

To get an idea of what needs to be generalized for a holomorphic correspondence,
we give here a brief summary of the Atiyah--Bott result for a holomorphic vector bundle.  For more details,
consult \cite[Section 4, pp.~455--459]{atiyah--bott68}.
Let $E$ be a holomorphic vector bundle over a compact complex manifold
$M$ and $f\colon M \to M$ a holomorphic map.
Denote by $\Gamma(E)$ the vector space of $\cinf$ sections of $E$ over
$M$
and by $\Lambda^{p,q}$ the $\cinf$ vector bundle of $(p,q)$-covectors
on $M$.
The smooth sections of $E \otimes \Lambda^{p,q}$ are the $E$-valued
$(p,q)$-forms on $M$.
The $\bar{\partial}$-operator on $(p,q)$-forms extends to $E$-valued
$(p,q)$-forms by acting as the identity on $E$ and as $\bar{\partial}$
on the forms.
There is then a differential complex
\[
  \Gamma(E) \overset{\dbar}{\longrightarrow} \Gamma(E \otimes \Lambda^{0,1})
\overset{\dbar}{\longrightarrow} \Gamma(E \otimes \Lambda^{0,2})
\overset{\dbar}{\longrightarrow} \cdots.
\]
The cohomology $H^*\big(M, \calo(E)\big)$ of $M$ with coefficients in
$E$ is defined to be the cohomology of this complex of $E$-valued
$(p,q)$-forms.

Now let $F$ be a holomorphic vector bundle over the complex manifold
$M$ and let $f^*F$ be its pullback under the holomorphic map $f\colon
M \to M$.
The map $f\colon M \to M$ induces a linear map of $\cinf$ sections
$f^*\colon \Gamma(F) \to \Gamma(f^*F)$ by sending a section $s\in
\Gamma(F)$ to
\[
  (f^*s)(x) = (s \comp f)(x) = s\big( f(x) \big) \in F_{f(x)} =
  (f^*F)_x, \quad x \in M
\]
where $F_{f(x)}$ is the fiber of $F$ at $f(x)$.
In order to obtain an endomorphism of $\Gamma(F)$, Atiyah and Bott
introduced the notion of a \term{lifting} of the map $f$ to the bundle
$F$.
It is a holomorphic bundle map $\varphi\colon f^*F \to F$ over $M$.
A lifting $\varphi$ induces a linear map $\varphi_*\colon \Gamma(f^*F) 
\to \Gamma(F)$ by composition: $\varphi_*(s) = \varphi \comp s$.
The holomorphic map $f\colon M \to M$ and a lifting $\varphi\colon
f^*F \to F$ together define an endomorphism of $\Gamma(F)$:
\[
  \Gamma(F) \overset{f^*}{\longrightarrow} \Gamma(f^*F) \overset{\varphi_*}{\longrightarrow}
  \Gamma(F).
\]
Applied to $F = E \otimes \Lambda^{0,k}$, this will then induce an
endomorphism
\[
  (f,\varphi)^*\colon H^*\big(M, \calo(E)\big) \to H^*\big(M,
  \calo(E)\big)
\]
and the Lefschetz number of the triple $(f, \varphi, E)$ is defined to be
\begin{equation} \label{7e:number}
  L(f, \varphi, E) \colon= \sum_{k=1}^n (-1)^k \tr \left. (f,\varphi)^*
  \right|_{H^k(M, \calo(E))}, \quad n=\dim_{\C} M.
\end{equation}

\begin{thm}\cite[Theorem 4.12, p.~458]{atiyah--bott68} \label{t:bundle}
  Let $E$ be a holomorphic vector bundle over a compact complex
  manifold $M$, $f\colon M \to M$ a transversal holomorphic
  self-map, and $\varphi\colon f^*E \to E$ a holomorphic bundle map.
  Then
  \[
    L(f, \varphi, E) = \sum_{f(p)=p} \frac{\tr \varphi_p}{\det (1-
      f_{*,p})}.
  \]
\end{thm}
In this theorem, a \term{transversal} map is one whose graph intersects the diagonal transversally
in $M \times M$,
$\varphi_p\colon E_{f(p)}= E_p \to E_p$ is a complex
linear map, and $f_{*,p}$ is the differential of $f$ on the holomorphic
tangent space of $M$ at $p$.

For a holomorphic correspondence, a plausible conjecture should have
the same form as Theorem~\ref{t:bundle}.

\begin{conjecture}
  {\it Let $E$ be a holomorphic vector bundle over a compact complex
    manifold $M$, $\Gamma \subset M \times M$ a transversal
    correspondence, and $\varphi\colon f^*E \to E$ a holomorphic
    bundle map.  Then $(\Gamma, \varphi)$ induces an endomorphism of
    $H^k\big(M, \calo(E)\big)$ for each $k$ such that the Lefschetz
    number 
    $L(\Gamma, \varphi, E)$ defined as in \eqref{7e:number}
    satisfies
      \[
    L(\Gamma, \varphi, E) = \sum_{f(p)=p} \frac{\tr \varphi_p}{\det \big(1-
      J(\Gamma)_p\big)},
  \]
   where $J(\Gamma)_p$ is the Jacobian matrix of $\Gamma$ at $p$.
}
\end{conjecture}

In Shimura's emails to Michael Atiyah and Loring Tu in June 2013 (see
Appendix), he actually claimed more.
He said he had conjectured at Woods Hole in 1965 a Lefschetz fixed
point formula for an algebraic correspondence between two holomorphic
vector bundles on two algebraic varieties of the same dimension.
The statement of this forgotten conjecture remains a mystery.

Stated more generally, Shimura's intention might have been the following
(as formulated by Mark Goresky in a recent private communication):

\begin{quote}
{\it Find and prove a holomorphic Lefschetz fixed point theorem that can be
used to calculate the trace of a Hecke correspondence on the
holomorphic cohomology, coherent cohomology, or $\dbar$-cohomology,
of a Hermitian locally symmetric space.}
\end{quote}

     \section*{Appendix}

     \noindent
     \textbf{Email from Goro Shimura to Loring Tu, June 13, 2013}

     \bigskip
     \noindent
     Dear Loring,

      It is nice to hear from you.  I remember that you sent me your book in
collaboration with Bott. Here is my belated thanks for the book!

      As for that fixed point formula I can say the following.

      In the case of Riemann surfaces, Eichler's result is quite general,
and so it was definitely meaningless to conjecture something only for Riemann
surfaces.

      What I conjectured was a formula for an algebraic correspondence,
not just for a map, between two algebraic varieties of the same
dimension, so that it generalizes Eichler's formula. (Naturally, we have
to (I had to) formulate it in terms of holomorphic bundles.) I thought
it might be applicable to automorphic forms on the higher-dimensional
spaces.

\begin{center}
  $\vdots$
\end{center}

      As I understand it, the Atiyah--Bott formula deals with only a map, not
a correspondence, and so it does not include Eichler's formula, nor
does it prove my conjecture. Therefore I think it is an open problem
to prove it for a correspondence. Am I wrong? 

\begin{center}
  $\vdots$
\end{center}

\noindent
 With best regards,\\
 Goro Shimura

 \bigskip
  \noindent
     \textbf{Email from Goro Shimura to Michael Atiyah, June 19, 2013}

     \bigskip
     \noindent
     Dear Michael,

     \begin{center}
  $\vdots$
\end{center}

           Frankly I am incapable of telling you what exactly my conjecture was. Probably I made notes, but I don't think I can find them.

I can tell you that it concerned an algebraic correspondence between two
holomorphic bundles on two base algebraic varieties of the same dimension,
consistent with an algebraic correspondence on the base varieties.
I formulated it so that it becomes Eichler's formula in the
 one-dimensional case, and  also it becomes
a special case of the Lefschetz fixed point formula when the bundles
 are trivial. I was not considering real analyticity. 

 \begin{center}
  $\vdots$
\end{center}
     
\noindent
With very best regards, \\
Goro

\end{document}